\documentclass[a4paper,11pt]{amsart}
\usepackage{eucal}
\usepackage{cite}
\usepackage{amsmath,amsthm,amssymb}
\usepackage{amsfonts}
\usepackage{latexsym}
\usepackage{hyperref}
\usepackage{eucal}
\usepackage{mathrsfs}
\usepackage{txfonts}
\theoremstyle{plain}
\newtheorem{theorem}{Theorem}[section]
\newtheorem{proposition}[theorem]{Proposition}
\newtheorem{lemma}[theorem]{Lemma}
\newtheorem{corollary}[theorem]{Corollary}

\theoremstyle{definition}
\newtheorem{definition}[theorem]{Definition}
\newtheorem{example}[theorem]{Example}
\makeatletter
\renewenvironment{proof}[1][\proofname]{\par
  \normalfont
  \topsep6\p@\@plus6\p@ \trivlist
  \item[\hskip\labelsep{\bfseries #1}\@addpunct{\bfseries.}]\ignorespaces
}{%
  \endtrivlist
}
\renewcommand{\proofname}{proof}
\theoremstyle{remark}
\newtheorem{remark}[theorem]{Remark}

\numberwithin{equation}{section}

\title{Motivic classifying $\infty$-topoi and spectral stacks}
\date{\today} 
\author{Yuki Kato}
\thanks{The author was supported by Grants-in-Aid for Scientific
Research No.26610010 and 23K03080, Japan Society for the Promotion of Science.}
\address{National institute of technology, Kurume college, 
	      1-1-1, Komorino, Kurume, Fukuoka, 830-8555 JAPAN.}
\email{kato\_051@kurume-nct.ac.jp}
\subjclass[2020]{18G55 (primary), 14F42, 14A30, 18N60 (secondary)}
\keywords{Model categories, $\A^1$-homotopy theory, Derived algebraic geometry.}

\newcommand{\A}{\mathbb{A}}

\newcommand{\E}{\mathbb{E}}

\newcommand{\proj}{\mathbb{P}}

\newcommand{\MSpec}{\mathrm{MSpc}}
\newcommand{\Hom}{\mathrm{Hom}}

\newcommand{\Mot}[2]{\mathrm{Mot}^{#2}_{#1}}
\newcommand{\sSet}{\mathrm{Set}_{\Delta}}
\newcommand{\sCat}{\mathrm{Cat}_{\Delta}}

\newcommand{\Set}[1]{\mathrm{Set}_{#1}}
\newcommand{\Cat}[1]{\mathrm{Cat}_{#1}}
\newcommand{\MCat}[1]{\mathrm{MCat}_{#1}}

\newcommand{\MMod}{\mathfrak{Mod}}
\newcommand{\MModT}{\mathfrak{Mod}^\otimes}
\newcommand{\MPMod}{\mathfrak{PMod}}
\newcommand{\MPModT}{\mathfrak{PMod}^\otimes}
\newcommand{\Sym}{\mathrm{Sym}}

\newcommand{\Map}{\mathrm{Map}}

\newcommand{\Mod}{\mathrm{Mod}}
\newcommand{\CAlg}{\mathrm{CAlg}}

\newcommand{\Fun}{\mathrm{Fun}}
\newcommand{\PMod}{\mathrm{PMod}}

\newcommand{\Sm}{\mathbf{Sm}}

\newcommand{\MS}{\mathbf{MS}}
\newcommand{\MSp}{\mathrm{MSp}}

\newcommand{\LMTop}{{}^{\mathrm{L}}\mathfrak{MTop}}

\newcommand{\RMTop}{{}^{\mathrm{R}}\mathfrak{MTop}}

\usepackage{enumerate}
\usepackage{array}
\usepackage[all]{xy} 
\usepackage{delarray}
\ifx\undefined\bysame
\newcommand{\bysame}{\leavemode\hbox to3em{\hrulefill}\,}
\fi

\begin{document}
\begin{abstract}
In this paper, we develop "motivic derived algebraic geometry," an enhancement of derived algebraic geometry for the $\mathbb{A}^1$-homotopy theory of Morel and Voevodsky. We formulate motivic versions of $\infty$-categories, $\infty$-topoi, spectral schemes, and spectral Deligne-Mumford stacks based on the theories of Joyal, Lurie, Toën, and Vezzosi. As the main application, we establish the existence of the motivic stackification functor. This functor is constructed as the left adjoint of the pullback along the geometric morphism between classifying motivic $\infty$-topoi. 
\end{abstract}

\maketitle
\section{Introduction}
\label{sec:introduction}

Derived algebraic geometry studies geometric objects by allowing sheaves
of functions to take values in homotopical or higher-categorical objects
rather than in ordinary sets.  In the work of Lurie~\cite{DAGV,DAGVII}
and To\"en--Vezzosi~\cite{HAGI}, spectral schemes and spectral
Deligne--Mumford stacks are formulated in terms of structured
$\infty$-topoi and their classifying $\infty$-topoi.  From this
viewpoint, the passage from affine or algebraic objects to stacks is
realized as a universal construction in the $\infty$-categorical setting.

The purpose of this paper is to develop a motivic analogue of this
framework.  We work over a Grothendieck site $\mathcal{X}$ equipped with
an interval object $I$; the principal example is the Nisnevich site of
smooth schemes over a base scheme, with an interval $\A^1$.  Given a left
proper combinatorial simplicial model category $\mathbf{M}$, we construct
a motivic model category $\Mot{\mathcal{X}}{I}(\mathbf{M})$ by imposing
both descent for the Grothendieck topology and invariance with respect
to the interval object.  This construction recovers the usual model
category of motivic spaces of Morel--Voevodsky~\cite{MV} when
$\mathbf{M}$ is the model category of simplicial sets, and it also
applies to models of $\infty$-categories and $\infty$-bicategories.

Using these motivic model categories, we define motivic
$\infty$-categories, motivic $\infty$-topoi, and motivic classifying
$\infty$-topoi.  The bicategorical part of this construction relies on
Lurie's model of $\infty$-bicategories, which utilizes scaled simplicial sets and 
the scaled straightening and unstraightening theorem~\cite[p.128,
Theorem 3.8.1]{GI}.  This theorem facilitates the transition between motivic
functors valued in $\infty$-categories and motivic locally coCartesian
fibrations over the corresponding base.
Since motivic $\infty$-topoi are defined fiberwise by ordinary
$\infty$-topoi, they inherit the universal-colimit property: colimits in
each fiber are stable under pullback and compatible with motivic
restriction functors.  This observation also motivates the treatment of
geometric objects defined by torsors.  The stability of torsors under
base change is a formal consequence of universality in an
$\infty$-topos, while the $\A^1$-invariance of their classification is a motivic condition that requires additional hypotheses on the
acting group object.

The main result of the paper is Theorem~\ref{spec}.  It states that a
geometric morphism
$f:\mathcal{K}\to\mathcal{K}'$ of motivic classifying $\infty$-topoi,
compatible with the relevant geometric structures, induces a pullback
functor
\[
 f^{-1}:\LMTop(\mathcal{K}')\to \LMTop(\mathcal{K})
\]
between the motivic $\infty$-categories of structured topoi, and that
this functor admits a left adjoint relative to the underlying motivic
$\infty$-topos.  Thus, the stackification procedure familiar from derived
algebraic geometry has a motivic counterpart.

As an application, we introduce motivic spectral schemes and motivic
spectral Deligne--Mumford stacks.  The motivic Zariski and motivic
\'etale classifying $\infty$-topoi are defined by imposing motivic
analogues of the usual local and strictly Henselian local conditions on
commutative algebra objects.  Applying Theorem~\ref{spec} to the
geometric morphism from the motivic discrete classifying topos to the
motivic \'etale classifying topos yields a functor from motivic
algebraic spaces to motivic stacks with the expected universal property.
In particular, any motivic algebraic space $\mathbb{X}$ is 
associated with a motivic stack $\mathbb{X}^{\text{\'et}}$.

The paper is organized as follows.  In Section~\ref{sec:MMC}, we define
motivic model categories and the associated motivic $\infty$-categories.
In Section~\ref{sec:MDAG}, we recall the scaled straightening and
unstraightening formalism, constructing motivic classifying
$\infty$-topoi; the proof of Theorem~\ref{spec} provided therein.  In
Section~\ref{sec:MAG}, we apply the theorem to motivic spectral schemes
and motivic spectral Deligne--Mumford stacks, concluding with the existence
of associated motivic stacks.

\section{Motivic model categories}
\label{sec:MMC}
In this section, we define motivic model categories associated with a
Grothendieck site $\mathcal{X}$ equipped with an interval object.  An
{\it interval} object $I$ of $\mathcal{X}$ is a triple
$(\mu : I \times I \to I,\, i_0,\,i_1: * \to I)$ satisfying
\[
 \mu \circ( i_0 \times \mathrm{id}) =
 \mu \circ( \mathrm{id} \times i_0 ) = i_0 \circ \mathbf{1}
\]
and
\[
 \mu \circ( i_1 \times \mathrm{id}) =
 \mu \circ( \mathrm{id} \times i_1 ) = \mathrm{id}_I,
\]
where $\mathbf{1}: I \to *$ denotes the canonical map.

\subsection{Definition of the motivic model structure of a left proper
combinatorial simplicial model category.}
\label{MSS}
Let $\mathcal{X}$ be a Grothendieck site with an interval object $I$.
We assume that $\mathcal{X}$ has enough points such that 
a morphism $f:X \to Y$ in $\mathcal{X}$ is an isomorphism if
$f_x:X_x \to Y_x$ is an isomorphism of sets for any point
$x:* \to \mathcal{X}$, where the functor
$(-)_x: \mathcal{X} \to \mathit{Sets}$ denotes the right adjoint of the
induced functor $x_*: \mathit{Sets} \to \mathcal{X}$.  A simplicial
object $U_\bullet : \Delta^{\rm op} \to \mathcal{X}$ with an augmentation
$\pi:U_\bullet \to X$ is a {\it hypercover} of $X$ if it satisfies the
following conditions:
\begin{itemize}
\item For any $n \ge 0$, $U_\bullet([n])$ is a coproduct of compact
objects represented by small objects of $\mathcal{X}$.
\item The augmentation $\pi:U_\bullet \to X$ is a stalkwise trivial Kan
fibration; that is, $\pi_x:U_{x,\,\bullet} \to *$ is a trivial Kan
fibration for any point $x:* \to X$.
\end{itemize}

Let $\mathbf{M}$ be a left proper combinatorial simplicial model
category, and let $\mathbf{M}^{\mathcal{X}^{\rm op}}$ denote the category
of $\mathbf{M}$-valued presheaves on $\mathcal{X}$, equipped with the
projective model structure.  This model structure is again left proper,
combinatorial, and simplicial.  Let
$\Mot{\mathcal{X}}{I}(\mathbf{M})$ denote the Bousfield localization of
$\mathbf{M}^{\mathcal{X}^{\rm op}}$ defined as follows.  The
cofibrations of $\Mot{\mathcal{X}}{I}(\mathbf{M})$ are the cofibrations
of $\mathbf{M}^{\mathcal{X}^{\rm op}}$.  An $\mathbf{M}$-valued presheaf
$F$ on $\mathcal{X}$ is {\it $\mathcal{X}$-local} if $F$ is fibrant in
$\mathbf{M}^{\mathcal{X}^{\rm op}}$ and the induced map
$F(\pi):F(X) \to F(|U_\bullet|)$ is a weak equivalence in $\mathbf{M}$
for any hypercover $\pi:U_\bullet \to X$ with $X \in \mathcal{X}$.
Here, the functor
$|-|:\Fun(\Delta^{\rm op},\,\mathcal{X}) \to \mathcal{X}$ denotes
the geometric realization of simplicial objects.  Furthermore, $F$ is
{\it $I$-local} if the canonical map $\mathbf{1}:I \to *$ induces a weak
equivalence $F(U) \to F(U \times I)$ in $\mathbf{M}$ for any
$U \in \mathcal{X}$.  We say that $F$ is
{\it motivic $\mathbf{M}$-local} if it is both $\mathcal{X}$-local and
$I$-local.  A map $f:F \to G$ of $\mathbf{M}$-valued presheaves on
$\mathcal{X}$ is a {\it motivic $\mathbf{M}$-equivalence} if the induced
map
\[
 f^*:\Hom(G,\,Z) \to \Hom(F,\,Z)
\]
is a weak homotopy equivalence of simplicial sets for any motivic local
$\mathbf{M}$-valued presheaf $Z$.  We call the resulting model structure
on $\Mot{\mathcal{X}}{I}(\mathbf{M})$ the {\it motivic
$\mathcal{X}$-model structure} of $\mathbf{M}$.

By \cite[p.56, Corollary 4.55]{MR2771591}, the projective model
structure on $\mathbf{M}^{\mathcal{X}^{\rm op}}$ is left proper,
combinatorial, and symmetric monoidal. Therefore, the iterated
Bousfield localization $\Mot{\mathcal{X}}{I}(\mathbf{M})$ of
$\mathbf{M}^{\mathcal{X}^{\rm op}}$ is also left proper.  Moreover, by
\cite[p.54, Proposition 4.47]{MR2771591}, the Bousfield localization
$\Mot{\mathcal{X}}{I}(\mathbf{M})$ is a symmetric monoidal localization
of $\mathbf{M}^{\mathcal{X}^{\rm op}}$. Hence, 
$\Mot{\mathcal{X}}{I}(\mathbf{M})$ is also a symmetric monoidal model
category:
\begin{theorem}
\label{M-Model}
Let $\mathbf{M}$ be a left proper combinatorial simplicial model
category, and let $\mathcal{X}$ be a Grothendieck site with an interval
object $I$.  Assume that $\mathcal{X}$ has enough points.  Then there is
a left proper combinatorial simplicial model structure on
$\Mot{\mathcal{X}}{I}(\mathbf{M})$ determined by the following data:
\begin{itemize}
 \item[(C)] Cofibrations are pointwise cofibrations.
 \item[(W)] Weak equivalences are motivic $\mathbf{M}$-weak equivalences.
 \item[(F)] Fibrations are morphisms having the right lifting property
             with respect to all morphisms that are both cofibrations
             and motivic $\mathbf{M}$-weak equivalences.
\end{itemize}
Furthermore, if $\mathbf{M}$ is a symmetric monoidal model category,
then $\Mot{\mathcal{X}}{I}(\mathbf{M})$ is also a symmetric monoidal
model category. \qed
\end{theorem}

Write $\MS_\infty^I=\Mot{\mathcal{X}}{I}(\sSet)_\infty$ and
$\MCat{\infty}^I=\Mot{\mathcal{X}}{I}(\sSet^+)_\infty$.  We refer to
$\MS_\infty^I$ as the $\infty$-category of {\it motivic spaces} and to
$\MCat{\infty}^I$ as the $\infty$-category of
{\it motivic $\infty$-categories}.  We now give a more explicit
description of motivic spaces and motivic $\infty$-categories.  By the
straightening and unstraightening theorem~\cite[p.74, Theorem 2.2.1.2]{HT},
we have a Quillen adjunction
\[
  \mathrm{St}_\mathcal{X}: (\sSet)_{/ N(\mathcal{X})}
  \rightleftarrows \sSet^{\mathcal{X}^{\rm op}} :
  \mathrm{Un}_{\mathcal{X}},
\]
where the model structure on the left-hand side is the contravariant
model structure~\cite[p.71, Remark 2.1.4.12]{HT}, and the model
structure on the right-hand side is the projective model structure
associated with the Kan--Quillen model structure on $\sSet$.  If $X$ is
a motivic space, then there exists a right fibration
$p_X:\overline{X} \to N(\mathcal{X})$ such that, for any
$U \in \mathcal{X}$, the fiber
$\overline{X} \times_{N(\mathcal{X})} U$ is homotopy equivalent to the
Kan complex $X(U)$ and satisfies $X(U \times I)\simeq X(U)$.
Similarly, a motivic $\infty$-category $\mathcal{C}$ is represented by
an $\sSet^+$-valued presheaf on $\mathcal{X}$ satisfying
$\mathcal{C}(U \times I) \simeq \mathcal{C}(U)$ for any
$U \in \mathcal{X}$.  In addition, $\MCat{\infty}^I$ is the full
subcategory of
$\Fun(\mathcal{X}^{\rm op},\,\Cat{\infty})$ spanned by $I$-local
objects.  By the straightening and unstraightening theorem for marked
simplicial sets with the Cartesian model structure~\cite[p.169,
Theorem 3.2.0.1]{HT}, we have a Quillen adjunction
\[
  \mathrm{St}^+_\mathcal{X}: (\widehat{\sSet^+})_{/ N(\mathcal{X})}
  \rightleftarrows (\widehat{\sSet^+})^{\mathcal{X}^{\rm op}}
  : \mathrm{Un}^+_{\mathcal{X}}
\]
between left proper combinatorial simplicial model categories.  Thus
there exists a Cartesian fibration
$p_\mathcal{C}:\overline{\mathcal{C}} \to N(\mathcal{X})$ such that,
for any $U \in \mathcal{X}$, the fiber
$\overline{\mathcal{C}} \times_{N(\mathcal{X})} U$ is weakly equivalent
to the $\infty$-category $\mathcal{C}(U)$ and satisfies
$\mathcal{C}(U \times I) \simeq \mathcal{C}(U)$.

\subsection{An $\infty$-category structure of motivic $\infty$-categories}

For any left proper combinatorial Cartesian closed simplicial monoidal
model category $\mathbf{M}$, the constant functor
\[
      \mathbf{M} \to \Mot{\mathcal{X}}{I}(\mathbf{M})
\]
is a left Quillen functor.  Its right adjoint is the global section
functor, namely the limit over $\mathcal{X}$.  Writing the constant
functor as $-\otimes \mathbf{1}$, we obtain the Quillen adjunction
\[
  - \otimes \mathbf{1}  :\mathbf{M}  \rightleftarrows
 \Mot{\mathcal{X}}{I}(\mathbf{M})  :
 \Map_{\widehat{\MCat{\infty}}}( \mathbf{1},\, - ).
\]
Consider the case $\mathbf{M}= (\sSet^+)_{/\Delta^0}$.  Let
$\MCat{\infty}$ denote the underlying $\infty$-category of
$\Mot{\mathcal{X}}{I}(\mathbf{M})$.

\subsection{Motivic $\infty$-categories as simplicial category functors}

\begin{definition}[\cite{HT} p.20, Definition \texorpdfstring{$1.1.5.3$}{1.1.5.3}]
Let $n$ be a non-negative integer.  For the standard $n$-simplex
$\Delta^n$, we define the simplicial category $\mathfrak{C}[\Delta^n]$
as follows:
\begin{itemize}
 \item The objects of $\mathfrak{C}[\Delta^n]$ are the objects of the
       category $[n]=\{0 \to 1 \to \cdots \to n \}$.
 \item For $0 \le i,\,j \le n$, the set of morphisms is defined by
\[
\Hom_{\mathfrak{C}[\Delta^n]}(i,\,j)=\begin{cases}
\emptyset &(j<i), \\
N(P_{i,\,j}) &(i \leq j),
\end{cases}
\]
where $P_{i,\,j}$ denotes the partially ordered set
$\{I \subset [n] \mid i, j \in I, I \subset \{i,\,\ldots , j \} \}$,
ordered by inclusion.
\end{itemize}
\end{definition}
The functor
\[
  \Delta \times \mathcal{X} \ni ([n],\,U) \mapsto
  (i_U)_*( \mathfrak{C}[\Delta^n]) \in (\sCat)^{\mathcal{X}^{\rm op}},
\]
where $(i_U)_*( \mathfrak{C}[\Delta^n]) : \mathcal{X}^{\rm op} \to
\sCat$ is the left Kan extension of
$\mathfrak{C}[\Delta^n]:\{U\} \to \sCat$ along the inclusion
$i_U:\{U\} \to \mathcal{X}$, induces an adjunction
\[
  \mathfrak{C}_\mathcal{X}: \Set{\Delta_\mathcal{X}}
  \rightleftarrows (\sCat)^{\mathcal{X}^{\rm op}} :N_{\Delta_\mathcal{X}}.
\]
The right adjoint $N_\Delta(\mathcal{C})$ of a functor
$\mathcal{C}:\mathcal{X}^{\rm op} \to \sCat$ is defined by
\[
 \Hom_{ \Set{\Delta_\mathcal{X}}}(\Delta^n \times U , N_\Delta (\mathcal{C})) =
\Hom_{ (\sCat)^{\mathcal{X}^{\rm op}}}( (i_U)_*( \mathfrak{C}[\Delta^n]),\, \mathcal{C}) = \Hom_{ (\sCat)^{\mathcal{X}^{\rm op}}}( \mathfrak{C}[\Delta^n],
\mathcal{C}(U))
\]
for each $n \ge 0$.

The category $\Cat{\Delta_\mathcal{X}}$ carries the model structure
induced by the motivic model structure on $\Set{\Delta_\mathcal{X}}$.
The motivic model structures on $\Set{\Delta_\mathcal{X}}$ and
$(\sCat)^{\mathcal{X}^{\rm op}}$ are chosen so that the functors
\[
    (\sCat)^{\mathcal{X}^{\rm op}}  \overset{N_\Delta}{\leftarrow}
    \Set{\Delta_\mathcal{X}} \overset{(-)^\flat}{\rightarrow}
    (\sSet^+)^{\mathcal{X}^{\rm op}}
\]
are left Quillen functors.
\begin{definition}
Let $(-)^\flat: \sSet \to \sSet^+$ denote the functor defined by
$X^\flat=(X,\,s_0(X))$, where $s_0(X)$ is the collection of degenerate
edges of $X$.  Let $f:F \to G$ be a morphism of simplicial sheaves on
$\mathcal{X}$.
\begin{itemize}
 \item[(C)] We say that $f$ is a {\it motivic cofibration} if
 $f^\flat:F^\flat \to G^\flat$ is a motivic cofibration in
 $(\sSet^+)^{\mathcal{X}^{\rm op}}$.
 \item[(W)] We say that $f$ is a {\it motivic equivalence} if
 $f^\flat:F^\flat \to G^\flat$ is a motivic equivalence in
 $(\sSet^+)^{\mathcal{X}^{\rm op}}$.
\end{itemize}
\end{definition}

\begin{definition}
Let $\mathfrak{C}_\mathcal{X}: \Set{\Delta_\mathcal{X}} \to
(\sCat)^{\mathcal{X}^{\rm op}}$ denote the functor defined above.
Let $f:\mathcal{C} \to \mathcal{D}$ be a natural transformation between
$\sCat$-valued presheaves on $\mathcal{X}$.
\begin{itemize}
 \item[(W)] We say that $f$ is a {\it motivic equivalence} if
 $N_{\Delta_\mathcal{X}}(f)$ is a motivic equivalence in
 $\Set{\Delta_\mathcal{X}}$.
 \item[(F)] We say that $f$ is a {\it motivic fibration} if
 $N_{\Delta_\mathcal{X}}(f)$ is a motivic fibration in
 $\Set{\Delta_\mathcal{X}}$.
\end{itemize}
\end{definition}

\begin{theorem}
The displayed left Quillen functors
\[
   (\sCat)^{\mathcal{X}^{\rm op}}  \overset{N_\Delta}{\leftarrow}
   \Set{\Delta_\mathcal{X}} \overset{(-)^\flat}{\rightarrow}
   (\sSet^+)^{\mathcal{X}^{\rm op}}
\]
are Quillen equivalences.
\end{theorem}
\begin{proof}
This follows from results of Bergner and Lurie~\cite[p.89,
Theorem 2.2.5.1; p.164, Theorem 3.1.5.1; and p.826,
Remark A.2.8.6]{HT}. \qed
\end{proof}

\subsection{The Yoneda lemma for motivic $\infty$-categories}

Let $\mathcal{C}$ be a motivic $\infty$-category, and let
$p:\overline{\mathcal{C}} \to N(\mathcal{X})$ be the Cartesian fibration
associated with $\mathcal{C}$.  For each $n \ge 0$, an
$N(\mathcal{X})$-morphism
$\Delta^n \times N(\mathcal{X}) \to \overline{\mathcal{C}}$ determines
an $n$-simplex in the fiber $\mathcal{C}(U)$ for any object
$U \in \mathcal{X}$.  Let $\mathrm{Ob}(\mathcal{C})$ denote the set of
$N(\mathcal{X})$-morphisms
$\Delta^0 \times N(\mathcal{X}) \to \overline{\mathcal{C}}$, and let
$\mathrm{Ar}(\mathcal{C})$ denote the set of morphisms
$\Delta^1 \times N(\mathcal{X}) \to \overline{\mathcal{C}}$.  We call
their elements the motivic objects and motivic morphisms of
$\mathcal{C}$, respectively.  The constant $\Cat{\infty}$-valued sheaf
$\mathcal{S}_\infty$ on $\mathcal{X}$ determines the constant Cartesian
fibration $\mathcal{S}_\infty \times N(\mathcal{X}) \to N(\mathcal{X})$.
We regard $N(\mathcal{X}) \times \mathcal{S}_\infty$ as the motivic
$\infty$-category of spaces.

\begin{definition}
\label{mapping}
Let $\mathcal{C}$ be a motivic $\infty$-category.  For two objects
$x,\,y: \Delta^0\times N(\mathcal{X}) \to \overline{\mathcal{C}}$, the
mapping space $\Map_\mathcal{C}(x,\,y)$ is defined to be the homotopy
limit
\[
     \{x\}   \times  N(\mathcal{X}) \times_{ \overline{\mathcal{C}} }   \mathrm{Ar}(\mathcal{C}) \times_{ \overline{\mathcal{C}} }   \{y\}   \times  N(\mathcal{X}).
\]
\end{definition}

We now define the Yoneda functor for motivic $\infty$-categories.  A
motivic $\infty$-category is said to be small if it is termwise small.
Let $p:\overline{\mathcal{C}} \to N(\mathcal{X})$ be the Cartesian
fibration determined by a small motivic $\infty$-category $\mathcal{C}$.
For any $N(\mathcal{X})$-morphism
$q:K \to \overline{\mathcal{C}}$, the overcategory
$r:\overline{\mathcal{C}}_{/q} \to \overline{\mathcal{C}}$ is a
Cartesian fibration.  For any point
$y \times N(\mathcal{X}):\Delta^0 \times N(\mathcal{X}) \to
\overline{\mathcal{C}}$, the fiber of
$r:\overline{\mathcal{C}}_{/y \times N(\mathcal{X})} \to
\overline{\mathcal{C}}$ is a Cartesian fibration and, therefore, a motivic
space.  Therefore, by the unstraightening
\[
\mathrm{Un}_{ \overline{\mathcal{C}}}:
({\sSet^+}_{/\mathcal{X}})_{/\overline{\mathcal{C}}} \to
((\sSet^+)_{/\mathcal{X}})^{\mathcal{C}^{\rm op}},
\]
we obtain a functor
\[
  \mathrm{Un}_{ \overline{\mathcal{C}}}
  (\overline{\mathcal{C}}_{/y \times N(\mathcal{X})})
  : \mathcal{C}^{\rm op} \to {\sSet^+}_{/N(\mathcal{X})},
\]
whose values at points
$\Delta^0 \times N(\mathcal{X}) \to \overline{\mathcal{C}}$ are motivic
spaces. We define
$\mathbf{Y}(y)=\mathrm{Un}_{ \overline{\mathcal{C}}}
(\overline{\mathcal{C}}_{/y \times N(\mathcal{X})})$ as the
{\it (motivic) Yoneda functor} of
$y \times N(\mathcal{X}): \Delta^0 \times N(\mathcal{X}) \to
\overline{\mathcal{C}}$.

\begin{lemma}
\label{MYoneda}
Let $\mathcal{C}$ be a small motivic $\infty$-category.  Then, for any
functor $f: \mathcal{C}^{\rm op} \to (\sSet)_{/\mathcal{X}}$, the
following square
\[
 \xymatrix@1{
\mathcal{C} \ar[rrr]^f \ar[d]_{\mathbf{Y}} && &    {\sSet}_{/N(\mathcal{X})}  \ar[d]^{i} \\
  ( {\sSet}_{/N(\mathcal{X})} )^{\mathcal{C}^{\rm
op}} \ar[rrr]_{\Map_{({\sSet}_{/N(\mathcal{X})} )^{\mathcal{C}^{\rm
op}} }(-,\,f)} & & &    {\widehat{{\sSet}}}_{/N(\mathcal{X})}
}
\]
is homotopically commutative, where $i$ denotes the canonical inclusion
from small simplicial sets to large simplicial sets over
$N(\mathcal{X})$.
\end{lemma}
\begin{proof}
For any object $U \in \mathcal{X}$, the fiber of the square is
homotopically commutative by Lurie's strong Yoneda lemma~\cite[p.461,
Proposition 5.5.2.1]{HT}.  Hence the square itself is homotopically
commutative. \qed
\end{proof}

\begin{corollary}
For any pair of points
$x,\,y: \Delta^0 \times N(\mathcal{X}) \to \mathcal{C}$, the mapping
spaces $\Map_{\mathcal{C}}(x,\,y)$ and
\[
\Map_{({\sSet}_{/N(\mathcal{X})} )^{\mathcal{C}^{\rm op}} }
(\mathbf{Y}(x),\,\mathbf{Y}(y))
\]
are weakly equivalent motivic spaces.
\end{corollary}
\begin{proof}
Applying Lemma~\ref{MYoneda} to the case $f=\mathbf{Y}(y)$, for any
$x$ one obtains a chain of equivalences
\[
\Map_{({\sSet}_{/N(\mathcal{X})} )^{\mathcal{C}^{\rm op}} }
(\mathbf{Y}(x),\,\mathbf{Y}(y))
\simeq \mathbf{Y}(y)(x) \simeq \Map_{\mathcal{C}}(x,\,y).
\]
\qed
\end{proof}

\subsection{Motivic presentable $\infty$-categories and motivic
$\infty$-topoi}
Let $\widehat{\Cat{\infty}}$ denote the very large $\infty$-category of
large $\infty$-categories, and let $\widehat{\MCat{\infty}}$ denote the
very large $\infty$-category of large motivic $\infty$-categories.  Let
${}^{\rm L}\mathrm{Pr}$ denote the subcategory of
$\widehat{\Cat{\infty}}$ spanned by presentable $\infty$-categories and
colimit-preserving functors.  We define the very large $\infty$-category
${}^{\rm L}\mathrm{MPr}$ of presentable motivic $\infty$-categories by
the pullback
\[
\xymatrix@1{ {}^{\rm L}\mathrm{MPr} \ar[d] \ar[r] &
\Fun(\mathcal{X}^{\rm op},\,{}^{\rm L}\mathrm{Pr}) \ar[d] \\
\widehat{\MCat{\infty}} \ar[r] &
\Fun(\mathcal{X}^{\rm op},\,\widehat{\Cat{\infty}}). }
\]

\begin{proposition}
\label{UniPre}
Let $F:\mathcal{C} \to \mathcal{D}$ be a functor of motivic
$\infty$-categories.  Assume that $\mathcal{C}$ is small and that
$\mathcal{D}$ is presentable.  Let $y: \mathcal{C} \to
\Fun(\mathcal{C}^{\rm op},\, \mathfrak{MS}_\infty)$ denote the motivic
Yoneda embedding.  Then the induced functor
\[
y_*(F): \mathcal{D} \ni d \mapsto
\Map_{\mathcal{D}}(F(-),\, d) \in
\Fun(\mathcal{C}^{\rm op},\, \mathfrak{MS}_\infty)
\]
admits a left adjoint.
\end{proposition}
\begin{proof}
For any $U \in \mathcal{X}$, the presentability of $\mathcal{D}(U)$
implies that the restriction functor
\[
y_*(F(U)): \mathcal{D}(U) \to
\Fun(\mathcal{C}(U)^{\rm op},\,\mathcal{S}_\infty)
\]
admits a left adjoint, namely left Kan extension along $F(U)$.  These
fiberwise left Kan extensions are functorial in $U$ by the universal
property of left Kan extensions, and hence assemble to a left adjoint of
$y_*(F)$ in the motivic functor category. \qed
\end{proof}

Let ${}^{\rm L}\mathrm{Top}$ denote the very large $\infty$-category of
$\infty$-topoi and left-exact colimit-preserving functors.  We define
the very large $\infty$-category ${}^{\rm L}\mathrm{MTop}$ of motivic
$\infty$-topoi by the pullback
\[
\xymatrix@1{ {}^{\rm L}\mathrm{MTop} \ar[d] \ar[r] &
\Fun(\mathcal{X}^{\rm op},\,{}^{\rm L}\mathrm{Top}) \ar[d] \\
\widehat{\MCat{\infty}} \ar[r] &
\Fun(\mathcal{X}^{\rm op},\,\widehat{\Cat{\infty}}). }
\]

\begin{proposition}
\label{MTopUniversalColimits}
Let $\mathcal{T}$ be a motivic $\infty$-topos.  Then colimits in
$\mathcal{T}$ are universal fiberwise.  More precisely, for any
$U\in\mathcal{X}$, any morphism $x\to y$ in $\mathcal{T}(U)$, and
any small diagram $D:K\to\mathcal{T}(U)_{/y}$, the canonical map
\[
 \left(\operatorname*{colim}_{k\in K}D_k\right)\times_y x
 \longrightarrow
 \operatorname*{colim}_{k\in K}(D_k\times_y x)
\]
is an equivalence.  Moreover, for any morphism $f:U\to V$ in
$\mathcal{X}$, the restriction functor
$\mathcal{T}(f):\mathcal{T}(V)\to\mathcal{T}(U)$ preserves these
universal colimits.
\end{proposition}
\begin{proof}
By definition, $\mathcal{T}(U)$ is an $\infty$-topos for any
$U\in\mathcal{X}$.  Hence colimits in $\mathcal{T}(U)$ are universal by
the $\infty$-categorical Giraud theorem~\cite[Theorem 6.1.0.6]{HT}.
The transition functors of a motivic $\infty$-topos are left exact and
colimit-preserving, so they preserve both pullbacks and colimits.  This
proves the compatibility with motivic restriction functors. \qed
\end{proof}

\begin{remark}
The preceding proposition separates two notions that are often related
in motivic examples.  Let $G$ be a group object of $\mathcal{T}(U)$ and
let $P\to X$ be a $G$-torsor.  For any morphism $Y\to X$, the pullback
$P\times_XY\to Y$ is again a torsor, now under the pulled-back group
$G_Y=G\times Y$.  This is a formal base-change property: the torsor
condition is expressed by an effective epimorphism together with the
equivalence $G_X\times_XP\simeq P\times_XP$, where $G_X=G\times X$,
and these conditions are preserved by
pullbacks in an $\infty$-topos.  This should not be confused with
$\A^1$-invariance of the classification of $G$-torsors.  The latter is a
separate motivic representability condition, familiar from
Morel--Voevodsky theory and from affine representability results for
vector bundles and principal bundles~\cite{MV,MR3679884,MR3748687}.
\end{remark}

\begin{proposition}
\label{presheaf}
Let $\mathcal{C}$ be a small motivic $\infty$-category.  Assume that,
for any morphism $f:U \to V$ in $\mathcal{X}$, the induced functor
$\mathcal{C}(f):\mathcal{C}(V) \to \mathcal{C}(U)$ is left exact.  Then
the motivic $\infty$-category
\[
\Fun_{\widehat{\MCat{\infty}}}(\mathcal{C}^{\rm op},\,
\mathfrak{MS}_\infty)
\]
is a motivic $\infty$-topos.
\end{proposition}
\begin{proof}
For any $U \in \mathcal{X}$, the fiber
\[
\Fun_{\widehat{\Cat{\infty}}}(\mathcal{C}(U)^{\rm op},\,
\mathfrak{MS}_\infty(U)) \simeq
\Fun_{\widehat{\Cat{\infty}}}(\mathcal{C}(U)^{\rm op},\,
\mathcal{S}_\infty)
\]
is an $\infty$-topos.  The left exactness assumption on the transition
functors $\mathcal{C}(f)$ implies, by
\cite[p.324, Theorem 5.1.5.6 and p.559, Proposition 6.1.5.2]{HT}, that
the induced transition functors between these presheaf $\infty$-topoi
are left exact and preserve colimits.  Hence the displayed motivic
$\infty$-category is an object of ${}^{\rm L}\mathrm{MTop}$. \qed
\end{proof}

\section{Motivic derived algebraic geometry}
\label{sec:MDAG}

\subsection{Conventions on $\infty$-bicategories}
We use Lurie's model of $\infty$-bicategories by scaled simplicial
sets~\cite{GI}.  A scaled simplicial set is a pair
$\overline{X}=(X,\,T)$ consisting of a simplicial set $X$ and a
collection of thin $2$-simplices containing all degenerate
$2$-simplices; we write $\sSet^{\rm sc}$ for the category of scaled
simplicial sets.  Let $\sSet^+$ denote the category of marked simplicial
sets with the Cartesian model structure, and let $\sCat^+$ denote the
category of $\sSet^+$-enriched categories with the model structure of
\cite[Definition A.3.2.4]{HT}.  We use the scaled nerve adjunction
\[
  \mathfrak{C}^{\rm sc}: \sSet^{\rm sc} \rightleftarrows \sCat^+ :N^{\rm sc}
\]
constructed in~\cite[Definition 3.1.10]{GI}.  We also use the scaled
model structure on $\sSet^{\rm sc}$ and Lurie's scaled straightening and
unstraightening theorem, recalled below only in the forms needed for our
motivic constructions.

\subsection{Scaled straightening and unstraightening}
We recall the definition of the scaled straightening and unstraightening
functors from~\cite[Section 3]{GI}.  Let $\overline{S}=(S,\,T)$ be a
scaled simplicial set, and let $\mathcal{C}$ be a $\sSet^+$-enriched
category.  Given a functor
$\phi: \mathfrak{C}^{\rm sc}[\overline{S}] \to \mathcal{C}$, we define
the scaled straightening functor
$\mathrm{St}_\phi^{\rm sc}:\Set{\Delta/\overline{S}}^+ \to
(\sSet^+)^\mathcal{C}$ as follows.

\begin{definition}[\cite{GI} p.114, Definition 3.5.1]
Let $\overline{X}=(X,\,M)$ be a marked simplicial set.  Let
$T \subset X \times \Delta^1$ be the collection of all $2$-simplices
$\sigma$ with the following properties:
\begin{itemize}
  \item Under the projection $X \times \Delta^1 \to X$, the image of
        $\sigma$ is a degenerate $2$-simplex of $X$.
  \item For any $2$-simplex
        $\pi:\Delta^2 \overset{\sigma}\to X \times \Delta^1 \to
        \Delta^1$ satisfying $\pi^{-1}(\{0\}) = \Delta^{0,\,1}$, the
        restriction $\sigma|_{\Delta^{0,\,1}}$ determines a marked edge
        of $X$.
\end{itemize}
We define a scaled simplicial set $C(\overline{X})$ by
\[
 C(\overline{X}) = (X \times \Delta^1) \coprod_{(X \times
 \{0\})_\flat} \{ v \}_\flat.
\]
We call $C(\overline{X})$ the {\it scaled cone} of $\overline{X}$.
More generally, for any scaled simplicial set $\overline{S}$ and any
$\overline{X} \in (\sSet^{\rm sc})_{/\overline{S}}$, we set
$C_{\overline{S}}(\overline{X}) = C(\overline{X}) \coprod_{(X \times
\{1\})_\flat} \overline{S}$.  We call $C_{\overline{S}}(\overline{X})$
the {\it scaled cone} of $\overline{X}$ over $\overline{S}$.
\end{definition}

\begin{definition}[\cite{GI} p.115, Definition 3.5.4]
Let $\overline{S}$ be a scaled simplicial set, let $\mathcal{C}$ be a
$\sSet^+$-enriched category, and let
$\phi: \mathfrak{C}^{\rm sc}[\overline{S}] \to \mathcal{C}$ be a functor
of $\sSet^+$-enriched categories.  The {\it scaled straightening functor}
associated with $\phi$ is the functor
\[
\mathrm{St}_\phi^{\rm sc}: \Set{\Delta/\overline{S}}^+ \to
(\sSet^+)^\mathcal{C}
\]
defined on $\overline{X}$ by
\[
 (\mathrm{St}^{\rm sc}_\phi(\overline{X}))(C)=
 \Map_{ C_{\overline{S}}[\overline{X}]
\coprod_{\mathfrak{C}^{\rm sc}[\overline{S}]} \mathcal{C} }
 (v,\, C),
\]
for any $C \in \mathcal{C}$.
\end{definition}

\begin{remark}
The straightening functor
$\mathrm{St}^{\rm sc}_\phi:\Set{\Delta/\overline{S}}^+ \to
(\sSet^+)^\mathcal{C}$ is defined as a $\sSet^+$-enriched categorical
colimit.  Let $f:X \to S$ be a marked simplicial set over $S$.  Then the
straightening $\mathrm{St}^{\rm sc}_\phi(\overline{X})$ is equivalent to
the colimit of the diagram
\[
   j^{\rm op}\circ \phi \circ  \mathfrak{C}^{\rm sc}[F]: \mathfrak{C}^{\rm sc}
[ (X \times \Delta^1,\,T )\coprod_{X \times \{1\} }\overline{S} ] \to \mathfrak{C}^{\rm sc}[\overline{S}] \to \mathcal{C} \to (\sSet^+)^\mathcal{C},
\]
where $F$ is induced by $f$ and
$j: \mathcal{C} \to (\sSet^{+})^{\mathcal{C}^{\rm op}}$ denotes the
enriched Yoneda embedding.
\end{remark}

Since the scaled straightening functor $\mathrm{St}^{\rm sc}_\phi$
preserves all small colimits, it has a right adjoint by the adjoint
functor theorem; we denote this right adjoint by
$\mathrm{Un}^{\rm sc}_\phi$.  We call $\mathrm{Un}_\phi^{\rm sc}$ the
{\it scaled unstraightening functor} associated with $\phi$.  It is known
that the adjunction
$(\mathrm{St}^{\rm sc}_\phi,\, \mathrm{Un}_\phi^{\rm sc})$ is a Quillen
adjunction.  Here the model structure on $\Set{\Delta/\overline{S}}^+$
is the locally coCartesian model structure~\cite[p.74, Example 3.2.9]{GI},
and the model structure on $(\sSet^+)^{\mathcal{C}}$ is the projective
model structure~\cite[pp.823--824, Definition A.2.8.1 and Proposition
A.2.8.2]{HT}.

We recall the scaled straightening and unstraightening theorem.
\begin{theorem}[\cite{GI} p.128, Theorem 3.8.1]
\label{StUnst}
Let $\overline{S}$ be a scaled simplicial set, let $\mathcal{C}$ be a
$\sSet^+$-enriched category, and let
$\phi: \mathfrak{C}^{\rm sc}[\overline{S}] \to \mathcal{C}$ be a weak
equivalence of $\sSet^+$-enriched categories.  Then the Quillen adjunction
\[
 \mathrm{St}_\phi^{\rm sc}: \Set{\Delta/\overline{S}}^+
 \rightleftarrows
(\sSet^+)^{\mathcal{C}}:\mathrm{Un}_\phi^{\rm sc}
\]
is a Quillen equivalence. \qed
\end{theorem}

\subsection{Definition of $\infty$-bicategories}
We recall the model structure on $\sSet^{\rm sc}$ using the model
structure on $\sCat^+$.
\begin{definition}[\cite{GI} p.115, Definition 3.5.6]
\label{bicat-eq}
Let $f:\overline{X} \to \overline{Y}$ be a morphism of scaled simplicial
sets.  We say that $f$ is a {\it bicategorical equivalence} if the
induced functor
$\mathfrak{C}^{\rm sc}[f]: \mathfrak{C}^{\rm sc}[\overline{X}] \to
\mathfrak{C}^{\rm sc}[\overline{Y}]$ is a weak equivalence of
$\sSet^+$-enriched categories.
\end{definition}

\begin{theorem}[\cite{GI} p.143, Theorem 4.2.7]
\label{Ibi}
Let $\sSet^{\rm sc}$ denote the category of scaled simplicial sets.
Then $\sSet^{\rm sc}$ has a left proper combinatorial model structure
determined by the following data:
\begin{itemize}
 \item[(W)] The weak equivalences in $\sSet^{\rm sc}$ are bicategorical
          equivalences.
 \item[(C)] The cofibrations in $\sSet^{\rm sc}$ are monomorphisms.
 \item[(F)] The fibrations in $\sSet^{\rm sc}$ are the morphisms having
          the right lifting property with respect to all morphisms
          satisfying (W) and (C).
\end{itemize}
Furthermore, the Quillen adjunction
\[
   \mathfrak{C}^{\rm sc} :   \sSet^{\rm sc}   \rightleftarrows   \Cat{\Delta}^+ :  N^{\rm sc}
\]
is a Quillen equivalence.
\qed
\end{theorem}

We call the model structure on $\sSet^{\rm sc}$ in Theorem~\ref{Ibi} the
{\it scaled model structure}.
\begin{definition}[\cite{GI} p.145, Definition 4.2.8]
An $\infty$-bicategory is a fibrant object of $\sSet^{\rm sc}$ with
respect to the scaled model structure.
\end{definition}

\begin{definition}
The scaled nerve of the $\sSet^+$-enriched category $\sSet^+$ is an
$\infty$-bicategory of $\infty$-categories, denoted
$\mathbf{Cat}_\infty$.  Let ${}^{\rm L}\mathbf{Top}$ denote the
subcategory of the $\infty$-bicategory $\mathbf{Cat}_\infty$ spanned by
$\infty$-topoi, whose morphisms are colimit-preserving functors.
\end{definition}

\subsection{Motivic $\infty$-topoi and motivic classifying $\infty$-topoi}

\begin{definition}[{\rm cf.}~\cite{HT} p.369, Definition 5.2.8.8 (Joyal)]
Let $\mathcal{C}$ be a motivic $\infty$-category.  A factorization
system $(S_L,\,S_R)$ is a pair of collections of morphisms of
$\mathcal{C}$ satisfying the following axioms:
\begin{enumerate}
\item The collections $S_L$ and $S_R$ are closed under retracts.
\item The collection $S_L$ is left orthogonal to $S_R$.
\item For any morphism $h:X \to Z$ in $\mathcal{C}$, there are an
object $Y$ of $\mathcal{C}$ and morphisms $f:X \to Y$ and $g:Y \to Z$
such that $h=g \circ f$, $f \in S_L$, and $g \in S_R$.
\end{enumerate}
\end{definition}
Let $\mathcal{T}$ and $\mathcal{T}'$ be motivic $\infty$-topoi.  Let
$\Fun^*(\mathcal{T},\,\mathcal{T}')$ denote the full subcategory of
$\Fun(\mathcal{T},\,\mathcal{T}')$ spanned by functors
$f:\mathcal{T} \to \mathcal{T}'$ that admit geometric left adjoints.  A
motivic classifying $\infty$-topos is a motivic $\infty$-topos
$\mathcal{K}$ such that $\mathcal{K}(U)$ is a classifying
$\infty$-topos for any $U \in \mathcal{X}$, and such that the
transition morphisms $\mathcal{K}(U) \to \mathcal{K}(V)$ are compatible
with the geometric structures for any morphism $V \to U$ in
$\mathcal{X}$.

\begin{definition}[\cite{DAGV} p.27, Definition 1.4.3]
\label{crtop}
Let $\mathcal{K}$ be a motivic $\infty$-topos.  A geometric structure on
$\mathcal{K}$ is a factorization system
$(S_L^\mathcal{T},\, S_R^\mathcal{T})$ on
$\Fun^*(\mathcal{K},\,\mathcal{T})$ that depends functorially on
$\mathcal{T}$.  We call $\mathcal{K}$ a {\it classifying motivic
$\infty$-topos}, and we call a morphism in $S^{\mathcal{T}}_R$ a
{\it local morphism}.  For any classifying motivic $\infty$-topos
$\mathcal{K}$ and any motivic $\infty$-topos $\mathcal{T}$, we let
$\mathrm{Str}^{\rm loc}_{\mathcal{K}}(\mathcal{T})$ denote the subcategory
of $\Fun^*(\mathcal{K},\,\mathcal{T})$ with the same objects and with
morphisms given by local morphisms.  We call an object of
$\mathrm{Str}^{\rm loc}_{\mathcal{K}}(\mathcal{T})$ a
{\it $\mathcal{K}$-structured sheaf} on $\mathcal{T}$.  If a geometric
morphism $f:\mathcal{K} \to \mathcal{K}'$ of classifying motivic
$\infty$-topoi induces, for any motivic $\infty$-topos $\mathcal{T}$,
a pullback functor that carries local morphisms in
$\Fun^*(\mathcal{K}',\,\mathcal{T})$ to local morphisms in
$\Fun^*(\mathcal{K},\,\mathcal{T})$, then we say that {\it $f$ is
compatible with the geometric structures}.
\end{definition}

The scaled straightening and unstraightening adjunction
\[
   \mathrm{St}^{\rm sc}:
  ( \Set{\Delta}^+)_{/N(\mathcal{X})\times  {}^{\rm L} \mathbf{Top}^{\rm op}  }
   \rightleftarrows
  ( \Set{\Delta}^+)^{ \mathcal{X}^{\rm op} \times    \mathfrak{C}^{\rm sc}\left[
  {}^{\rm L} \mathbf{Top}  \right]}
   :\mathrm{Un}^{\rm sc}
\]
induces a Quillen equivalence
\[
   \Mot{\mathcal{X}}{I}(\mathrm{St}^{\rm sc}):
   (\Set{\Delta})^+_{/ N(\mathcal{X})\times  {}^{\rm L} \mathbf{Top}^{\rm op}  }
   \rightleftarrows
   \Mot{\mathcal{X}}{I}((\Set{\Delta}^+)^{\mathfrak{C}^{\rm sc}\left[{}^{\rm L} \mathbf{Top}      \right]})
   :\Mot{\mathcal{X}}{I}(\mathrm{Un}^{\rm sc})
\]
of left proper combinatorial simplicial model categories.
Let
\[
 \LMTop(\mathcal{K}): \mathcal{X} \ni U
 \mapsto \LMTop(\mathcal{K}(U)) \in
 N((\Set{\Delta}^+)^\circ)_{/{}^{\rm L}\mathbf{Top}^{\rm op}}
\]
denote the scaled unstraightening of the $\Cat{\infty}$-valued presheaf
\[
\mathrm{Str}^{\rm loc}_\mathcal{K}:
\mathcal{X}^{\rm op} \ni U \mapsto
\mathrm{Str}^{\rm loc}_{\mathcal{K}(U)}(-) \in
\Fun_{\mathrm{Cat}_{(\infty,\,2)}}
({}^{\rm L}\mathbf{Top},\,\mathbf{Cat}_{\infty}).
\]

Furthermore, the motivic Yoneda functor
\[
\Fun^*(\mathcal{K},\,- ):
N(\mathcal{X}) \times {}^{\rm L}\mathbf{Top} \to
\widehat{\mathbf{Cat}_{\infty}}
\]
classifies an $\infty$-category $\LMTop_{\mathcal{K}/}$ and a motivic
locally coCartesian fibration
$q:\LMTop_{\mathcal{K}/} \to N(\mathcal{X})\times{}^{\rm L}\mathbf{Top}$.
By an argument analogous to the proof of
\cite[p.610, Proposition 6.3.4.6]{HT}, the $\infty$-category $\RMTop$
admits pullbacks.  In other words, for any geometric morphism
$f:\mathcal{K} \to \mathcal{K}'$, the forgetful functor
$f_* : \RMTop_{/\mathcal{K}} \to \RMTop_{/\mathcal{K}'}$ admits a right
adjoint.  The opposite category of $\LMTop_{\mathcal{K}/}$ is weakly
equivalent to $(\RMTop_{/\mathcal{K}})^{\rm op}$ as a motivic
$\infty$-category.  Hence we have a homotopically commutative diagram of
$\infty$-categories:
\[
 \xymatrix@1{
 \LMTop_{ \mathcal{K} / } \ar[r]<0.5mm>^{f_*}  &
 \LMTop_{\mathcal{K}'/ } \ar[l]<0.5mm>^{f^{-1}}  \\
 \LMTop(\mathcal{K}) \ar[u]  & \LMTop(\mathcal{K}') \ar[l]^{f^{-1}} \ar[u]
}.
\]
We now prove that the lower horizontal functor has a left adjoint.
\begin{theorem}
\label{spec}
Let $f:\mathcal{K} \to \mathcal{K}'$ be a geometric morphism of motivic
classifying $\infty$-topoi compatible with the geometric structures.
Given the commutative diagram
\[
\xymatrix@1{
 \LMTop(\mathcal{K}) \ar[dr]_p & &  \ar[ll]^{f^{-1}}
 \LMTop(\mathcal{K}') \ar[dl]_q \\
 &   N(\mathcal{X})\times {}^{\rm L} \mathbf{Top}  &
}
\]
where $f^{-1}$ is the functor induced by $f$, and where $p$ and $q$ are
motivic locally coCartesian fibrations.  Then
$f^{-1}:\LMTop(\mathcal{K}') \to \LMTop(\mathcal{K})$ admits a left
adjoint relative to $N(\mathcal{X})\times {}^{\rm L}\mathbf{Top}$.
\end{theorem}
\begin{proof}
The unstraightening functor
\[
   \Mot{\mathcal{X}}{I}(\mathrm{Un}^{\rm sc})  :
   \Mot{\mathcal{X}}{I}((\Set{\Delta}^+)^{\mathfrak{C}^{\rm sc}\left[{}^{\rm L} \mathbf{Top}      \right]}) \to (\Set{\Delta})^+_{/  N(\mathcal{X}) \times   {}^{\rm L} \mathbf{Top}^{\rm op}  }
\]
carries fibrant objects to fibrant objects.  Hence the induced morphism
$f^{-1}:\LMTop(\mathcal{K}') \to \LMTop(\mathcal{K})$ is a morphism
between fibrant objects over
$N(\mathcal{X}) \times {}^{\rm L}\mathbf{Top}^{\rm op}$.  In particular,
$f^{-1}$ carries motivic locally $q$-coCartesian edges to motivic locally
$p$-coCartesian edges.  By \cite[Proposition 7.3.2.6]{HA}, it is
sufficient to prove that the functor
\[
 f^{-1}_{\mathcal{T}}: \mathrm{Str}^{\rm loc}_{\mathcal{K}'}
 (\mathcal{T})  \to  \mathrm{Str}^{\rm loc}_{\mathcal{K}} (\mathcal{T})
\]
admits a left adjoint for any motivic $\infty$-topos $\mathcal{T}$.
Let $\mathcal{O}_\mathcal{T}$ be an object of
$\Fun^*(\mathcal{K},\,\mathcal{T})$, and let $\mathcal{O}'_\mathcal{T}$
be an object of $\Fun^*(\mathcal{K}',\,\mathcal{T})$.  Let
$\phi:\mathcal{O}_\mathcal{T} \to \mathcal{O}'_\mathcal{T} \circ f$ be a
local morphism in $\Fun^*(\mathcal{K},\,\mathcal{T})$.  We obtain a left
Kan extension
$f_*\mathcal{O}_\mathcal{T}: \mathcal{K}' \to \mathcal{T}$ along $f$.
The transformation
$\phi_*: f_*(\mathcal{O}_\mathcal{T}) \to \mathcal{O}'_\mathcal{T}$
induced by $\phi$ gives a functorial factorization
\[
 f_*(\mathcal{O}_\mathcal{T}) \to \MSpec^{\mathcal{K}}_{\mathcal{K}'}
 (\mathcal{O}_\mathcal{T}) \overset{\MSpec(\alpha)}\to
 \mathcal{O}'_\mathcal{T}
\]
where $\MSpec(\alpha)$ is local.  Hence we obtain a functor
\[
 \MSpec^{\mathcal{K}'}_{\mathcal{K},\,\mathcal{T}}:
 f_*(\mathcal{O}_\mathcal{T}) \to \mathcal{O}'_\mathcal{T}
\]
which is a left Kan extension of $f^{-1}_\mathcal{T}$. \qed
\end{proof}

\begin{remark}
In this paper, the motivic $\infty$-category $\LMTop(\mathcal{K})$ is
constructed following~\cite[Remark 1.4.17]{DAGV}.
\end{remark}

\section{The definition of motivic spectral schemes and motivic (Deligne--Mumford) stacks.}
\label{sec:MAG}
Fix a regular Noetherian separated scheme $S$ of finite dimension, and
let $\Sm_S$ be the Nisnevich site of smooth schemes over $S$ with interval
object $\A^1$. We write $\mathfrak{MS}_\infty$ for the motivic
$\infty$-category of motivic spaces and $\mathfrak{MSp}_\infty$ for the
stable motivic $\infty$-category of motivic spectra. Let
$\mathfrak{MSp}_\infty^\omega$ denote the full subcategory of compact
motivic spectra.

Recall that Zariski topos is classified by the factorization system of
local morphisms between local rings, and \'etale topos is classified by
the factorization system of local morphisms between strict Henselian
local rings. We introduce the definition of local ring objects, local
morphisms and strict Henselian local ring objects of a motivic
$\infty$-topos.

\subsubsection{Motivic $\infty$-Zariski topos.}
Let $\mathcal{T}$ be a topos with a final object $\bf{1}$ and
$\mathcal{O}$ a commutative algebra object of $\mathcal{T}$. We say that
$\mathcal{O}$ is {\it local}~\cite[p.13, Definition 2.4]{DAGVII} if the
following conditions are satisfied:
\begin{enumerate}
\item Let $0, \, 1 : \bf{1} \to \mathcal{O}$ denote the additive
      identity and multiplicative identity in $\mathcal{O}$. Then
      $\bf{1} \times_\mathcal{O} \bf{1}$ is an initial object of
      $\mathcal{T}$.
  
\item Let $\mathcal{O}^\times$ be the multiplicative group of
      $\mathcal{O}$ which is given by the pullback square
\[
\xymatrix@1{
 \mathcal{O}^\times \ar[d] \ar[r]^e  & 
 \mathcal{O} \times \mathcal{O} \ar[d]  \\
  \bf{1} \ar[r] & \mathcal{O}, 
}
\]
where $m: \mathcal{O} \times \mathcal{O} \to \mathcal{O}$ is the
  multiplication equipped with $\mathcal{O}$. Then the map \[
  (1-e)\coprod e : \mathcal{O}^\times \coprod \mathcal{O}^\times \to
  \mathcal{O}
\] 
is an effective epimorphism (See \cite[p.531, Remark 6.1.1.5]{HT}). 

Let $\alpha: \mathcal{O} \to \mathcal{O}'$ be a morphism of commutative
      local ring objects of $\mathcal{T}$. We say that $\alpha :
      \mathcal{O} \to \mathcal{O}'$ is a {\it local morphism} if the
      diagram
      \[
      \xymatrix@1{ \mathcal{O}^\times \ar[r]^\alpha \ar[d] &
      \mathcal{O}'^\times \ar[d] \\ \mathcal{O} \ar[r]_\alpha &
      \mathcal{O}' }
	\]
is a pullback square.
\end{enumerate}
  
\begin{definition}[{\rm c.f.}~\cite{DAGVII} p.13, Definition 2.5]
\label{local} Let $\mathcal{O}$ be a commutative ring object of a
motivic $\infty$-topos on $\mathcal{T}$. We say that $\mathcal{O}$ is a
{\it local} if $\pi_0 \mathcal{O}(U)$ is a local ring object of $\mathrm{h}
 \mathcal{T} (U)$ for any $U \in \mathcal{T}$.
\end{definition}  
      
Let $\mathcal{K}^{M}_{disc}$ denote the motivic $\infty$-topos $\Fun(
\CAlg(\mathfrak{MSp}^\omega_\infty)^{\rm op},\,\mathfrak{MS}_\infty)$.
We introduce the classifying Zariski topos on $\mathcal{K}_{\rm
disc}^M$.  The motivic $\infty$-categorical Yoneda functor $y: \mathfrak{MSp}^{\omega}_\infty
\to \Fun ( \mathfrak{MSp}_\infty^{\omega,\,\rm op}
,\,\mathfrak{MS}_\infty)$ induces $\CAlg(y): \CAlg(
\mathfrak{MSp}^{\omega}_\infty) \to \CAlg(\Fun (
\mathfrak{MSp}_\infty^{\omega, \,\rm op}
,\,\mathfrak{MS}_\infty))$. Hence we obtain the canonical functor:
\[
\Fun(\CAlg(\mathfrak{MSp}^\omega_\infty)^{\rm op}, \,
\mathfrak{MS}_\infty ) \to \CAlg( \Fun (
\mathfrak{MSp}_\infty^{\omega,\, \rm op} ,\,\mathfrak{MS}_\infty) ).\]
It
is well-known that this canonical functor induces weak equivalences on
each fiber on $X \in \mathcal{X}$. Therefore the canonical functor is a
weak equivalence of motivic $\infty$-categories.  We let
$\mathcal{K}_{\rm Zar}^M$ denote the subcategory of $\mathcal{K}_{\rm
disc}^M$ whose morphisms are local morphisms. We say that
$\mathcal{K}_{\rm Zar}^M$ is the {\it motivic $\infty$-Zariski topos}.
Write $\mathrm{MSch}= \LMTop(\mathcal{K}_{\rm Zar}^M)^{\rm op}$.  A {\it
motivic scheme} is an object of the motivic $\infty$-category
$\mathrm{MSch}$.  Equivalently a motivic scheme is an $\A^1$-homotopy
invariant spectral scheme-valued Nisnevich-local sheaf on $\Sm_S$.

Write $\mathrm{MAlgSp}= \LMTop(\mathcal{K}_{\text{disc}}^M)^{\rm
op}$. We refer to $\mathrm{MAlgSp}$ as the motivic $\infty$-category of
{\it motivic algebraic spaces}.  Then the
geometric morphism $\mathcal{K}_{\text{disc}}^M \to
\mathcal{K}_{\text{Zar}}^M$ induces the functor
\[
 (-)^{\text{Zar}} : \mathrm{MAlgSp} \to \mathrm{MStk}		    
\]
which admits a left adjoint by Theorem~\ref{spec}.

\subsubsection{Motivic $\infty$-\'etale topoi.}
\begin{definition}[\cite{DAGVII} p.68, Definition 8.1]
\label{stH} Let $\mathcal{T}$ be a topos and $\mathcal{O}_\mathcal{X}$ a
commutative ring object of $\mathcal{T}$. For any finitely generated
algebra $R$, let $\mathrm{Sol}_R(\mathcal{O}_\mathcal{T})$ be an object
of $\mathcal{T}$ defined by
\[
  \mathrm{Sol}_R(\mathcal{O}_\mathcal{T}) : \mathcal{T} \ni U \mapsto
 \Hom_{\rm Ring}(R,\, \Hom_\mathcal{T}(U,\,\mathcal{O}_\mathcal{T}) )
 \in \mathrm{Sets}.
\] 
We say that $\mathcal{O}_\mathcal{T}$ is {\it strictly Henselian}, if
 for any finite collection of \'etale maps $R \to R_\alpha$ which
 induce a faithfully flat map $R \to \prod_\alpha R_\alpha$, the induced
 map
 \[
 \coprod_\alpha \mathrm{Sol}_{R_\alpha}(\mathcal{O}_\mathcal{T})
 \to \mathrm{Sol}_R (\mathcal{O}_\mathcal{T})
 \]  
is an effective epimorphism. 
\end{definition}

\begin{definition}[{\rm c.f.}~\cite{DAGVII} p.68, Definition 8.3]
\label{STH} Let $\mathcal{T}$ be a motivic $\infty$-topos and
$\mathcal{O}_\mathcal{T}$ a commutative algebra object of
$\mathcal{T}$. Then we say that $\mathcal{O}_\mathcal{T}$ is {\it
strictly Henselian} if $\pi_0 \mathcal{O}_\mathcal{T}(U)$ is a strictly
Henselian commutative ring object of the category $\mathrm{h}
\mathcal{T}(U)$ for each $U \in \mathcal{T}$. Let
$\alpha:\mathcal{O}_\mathcal{T} \to \mathcal{O}'_\mathcal{T}$ be a local
morphism of local commutative algebra objects of $\mathcal{T}$. We say
that $\alpha$ is a {\it strict Henselian local} if
$\mathcal{O}_\mathcal{T}$ and $\mathcal{O}'_\mathcal{T}$ are strict
Henselian.
\end{definition}

Let $\mathcal{K}_{\text{\'et}}^M$ denote the subcategory of
$\mathcal{K}_{\rm disc}^M$ whose morphisms are strict Henselian
local. We say that the motivic $\infty$-topos
$\mathcal{K}_{\text{\'et}}^M$ is the {\it motivic $\infty$-\'etale
topos}.  Write $\mathrm{MStk}= \LMTop(\mathcal{K}_{\text{\'et}}^M)^{\rm
op}$.  A {\it motivic stack} is an object of the motivic
$\infty$-category $\mathrm{MStk}$. By a similar argument to the case
of motivic schemes, a motivic stack is a $\A^1$-homotopy invariant
spectral Deligne--Mumford stack-valued Nisnevich-local sheaf. Similarly,
the geometric morphism $\mathcal{K}_{\text{disc}}^M \to
\mathcal{K}_{\text{\'et}}^M$ induces the functor
\[
 (-)^{\text{\'et}} : \mathrm{MAlgSp} \to \mathrm{MStk}		    
\]
which admits a left adjoint by Theorem~\ref{spec}. Let $\mathbb{X}$ be a motivic algebraic space. 
Then we say that $\mathbb{X}^{\text{\'et}}$ is the motivic stack associated to $\mathbb{X}$.

\begin{corollary}
Let $\mathrm{MAlgSp}$ and $\mathrm{MStk}$ denote the motivic $\infty$-categories of motivic algebraic spaces and motivic stacks, respectively. The geometric morphism $\mathcal{K}_{disc}^M \to \mathcal{K}_{\text{\'et}}^M$ induces a pullback functor $(-)^{\text{\'et}} : \mathrm{MAlgSp} \to \mathrm{MStk}$. This functor admits a left adjoint relative to the underlying $\infty$-topos. Consequently, for any motivic algebraic space $\mathbb{X}$, there universally exists an associated motivic stack $\mathbb{X}^{\text{\'et}}$.
\end{corollary}
\begin{proof}
This follows directly from Theorem~\ref{spec} and the definition of the motivic $\infty$-\'etale topos $\mathcal{K}_{\text{\'et}}^M$. The relative left adjoint provides the required universal property of the stackification. 
\end{proof}

\subsubsection*{Acknowledgements}  
The author used Google’s Gemini-pro 3.1 as a conversational research aid for brainstorming and refining mathematical formulations, and OpenAI’s Prism (GPT-5.2) for editorial and expository assistance, including structural refinement. These tools were used to improve the manuscript’s clarity and to indicate where abbreviated arguments warranted fuller exposition. All results, proofs, and final formulations were reviewed, verified, and approved by the author, who bears sole responsibility for the content of this work.
\bibliographystyle{alphadin}

\bibliography{bibkato}
\end{document}